\documentclass[preprint,sort&compress,review,3p]{elsarticle}
\usepackage{amssymb, amsmath}
\usepackage{array}
\usepackage{multirow}
\usepackage{moreverb}
\usepackage{graphicx,epsfig}
\usepackage{caption}
\usepackage{algorithm}
\usepackage{algorithmicx}
\usepackage{algpseudocode}
\usepackage{float}
\usepackage{stfloats}
\usepackage{float}
\usepackage{amsthm}
\usepackage{amsfonts}
\usepackage{amssymb}
\usepackage{color}
\usepackage{mathrsfs}
\usepackage{amsmath}
\usepackage{cases}

\newtheorem{theorem}{Theorem}

\newtheorem{lemma}{Lemma}

\newtheorem{assum}{Assumption}

\newtheorem{rem}{Remark}

\usepackage{amssymb}

\usepackage{lineno}
\begin{document}
\begin{frontmatter}
\title{ Two system transformation data-driven algorithms for linear quadratic mean-field games $^\dagger$}

\author[address1]{Xun Li}\ead{li.xun@polyu.edu.hk}     
\author[address2]{Guangchen Wang}\ead{wguangchen@sdu.edu.cn}    
\author[address2]{Yu Wang}\ead{wangyu1@mail.sdu.edu.cn}           
\author[address3,address4]{Jie Xiong}\ead{xiongj@sustech.edu.cn}
\author[address2]{Heng Zhang}\ead{zhangh2828@163.com,  zhangheng2828@mail.sdu.edu.cn}  

\address[address1]{Department of Applied Mathematics, The Hong Kong Polytechnic University, Hong Kong, China}                                               
\address[address2]{School of Control Science and Engineering, Shandong University, Jinan 250061, China}             
\address[address3]{Department of Mathematics, Southern University of Science and Technology, Shenzhen 518055, China}        
\address[address4]{SUSTech International Center for Mathematics, Southern University of Science and Technology, Shenzhen 518055, China}  




\begin{abstract}
 This paper studies a class of continuous-time linear quadratic (LQ) mean-field game problems. We develop two system transformation data-driven  algorithms to approximate the decentralized strategies of the LQ mean-field games. The main feature of the obtained data-driven algorithms is that they eliminate the requirement on all system matrices. First, we transform the original stochastic system into an ordinary differential equation (ODE). Subsequently, we construct some Kronecker product-based matrices by the input/state data of the ODE. By virtue of these matrices, we implement a model-based policy iteration (PI) algorithm and a model-based value iteration (VI) algorithm in a data-driven fashion. In addition, we also demonstrate the convergence of these two data-driven algorithms under some mild conditions.  Finally, we illustrate the practicality of our algorithms via two numerical examples.
\end{abstract}

\end{frontmatter}

\section{Introduction}

Game theory is the study of decision making in an interactive environment, which has been widely used in industry, management and other fields \cite{von Neumann and Morgenstern 1944,Baser1982,Gueant et al. 2011,Bensoussa et al. 2013}.  In game systems with a large number of agents, i.e., large-population systems, the influence of individual  agent behavior on the systems can be negligible, but the group behavior of agents has a significant impact on individual agent. In detail, states or cost functionals of agents are highly coupled through a state-average term, which brings curse of dimensionality and considerable computational complexity. Therefore,  the classical game theory will no longer be applicable for large-population systems.

Different from the classical game theory, Huang et al. \cite{Huang et al. 2006} and Lasry and Lions \cite{Lasry and Lions 2007} independently proposed a mean-field method to overcome difficulties caused by high coupling. Huang et al. \cite{Huang et al. 2006} studied collective behavior caused by individual interactions and utilized a mean-field term to represent the complex interaction information between agents and proposed decentralized strategies. Independently, the method proposed by Lasry and Lions \cite{Lasry and Lions 2007} entailed solving coupled forward-backward partial differential equations. Since then, mean-field game theory has developed rapidly and the game framework has been extended to various different settings. For instance, Huang et al. \cite{Huang et al. 2007} considered an LQ game system where each agent is weakly coupled with the other agents only through its cost functional. Based on a fixed-point method, they developed the Nash certainty equivalence and designed an $\epsilon$-Nash equilibrium for the infinite-horizon mean-field LQ game. Huang \cite{Huang 2010} further considered decentralized controls of a game system with a major player and a large number of minor players, where the major player has a significant influence on others. Utilizing a mean-field approximation, they decomposed the game problem in population limit into a set of localized limiting two-player games and derived all players' decentralized strategies by the Nash certainty equivalence approach. Readers may also refer to \cite{Bensoussan et al. 2016,Bensoussan et al. 2017,Moon and Basar 2018} for mean-field Stackelberg differential games, \cite{Moon and Basar 2017,Tembine et al. 2011} for risk-sensitive mean-field games, \cite{Huang and Wang 2015,Huang et al. 2023} for mean-field differential games with partial information, \cite{Wang et al. 2020} for a mean-field cooperative differential game.

On the other hand, reinforcement learning (RL), as a well-known numerical algorithm for studying LQ control and game problems, has received increasing attention from researchers. The main feature of RL algorithms is that they do not require or partially require the system matrix information, and thus may have higher practical value. See \cite{Vrabie2009,JiangJiangPI2012,BianJiangVI2016,Lee2012} for some classical RL algorithms related to deterministic LQ control problems and \cite{LiXu2020,ZhangLi2023,Zhang2023,WangZhang2022} for some RL algorithms in stochastic LQ control problems. For mean-field game problems, uz Zaman et al. \cite{uzZaman2023} designed a stochastic optimization-based RL algorithm to solve a discrete-time LQ mean-field game problem. Fu et al. \cite{Fuetal2020} established an actor-critic RL algorithm for an LQ mean-field game in discrete time. Subramanian and Mahajan \cite{SubramanianMahajan2019} proposed two RL algorithms to deal with a stationary mean-field game problem. Carmona et al. \cite{CarmonaLau2019PolicyGradient} obtained a policy gradient strategy to work out a mean-field control problem and Angiuli et al. \cite{Angiuli} introduced a unified Q-learning for mean-field game and control problems. Carmona et al. \cite{CarmonaLau2023} developed a Q-learning algorithm for mean-field Markov decision processes and presented its convergence proof.  However, it is worth mentioning that the aforementioned RL literature focuses on LQ control problems or discrete-time mean-field game problems and there are few RL algorithms for continuous-time LQ mean-field games. To the authors' best knowledge, Xu et al. \cite{XuShenHuang2023} is the first attempt to solve continuous-time LQ mean-field game problems by using data-driven RL algorithms. The $\epsilon$-Nash equilibrium of their problem is closely related to two algebraic Riccati equations (AREs). They first designed a model-based PI algorithm to solve these two AREs, where all system coefficients are indispensable. Then, they implemented the model-based PI algorithm by collecting the state and input data of a given agent, and thus removed the requirement of all system parameters. 

Inspired by the above articles, this paper focuses on the same  LQ mean-field game problem as in Xu et al. \cite{XuShenHuang2023}, but proposes two novel data-driven RL methods to address this problem.  In particular, we provide a different idea to implement the model-based PI algorithm and a model-based VI algorithm, which reduces the computational complexities of our data-driven algorithms. The main contributions of this paper are summarized as follows.
\begin{itemize}
	\item We propose a system transformation idea to implement the model-based algorithms. Specifically, we transform the original stochastic system into an ODE and then carry out the model-based algorithms by the input/state data of the ODE. Different from the algorithms that directly use stochastic data \cite{LiXu2020,ZhangLi2023,Zhang2023,WangZhang2022,XuShenHuang2023}, the algorithms that adopt our idea have smaller computational complexities. In addition, this idea may also be applicable to other data-driven RL algorithms for stochastic problems,  especially for problems where the diffusion term of their system dynamics does not include control and state variables. 
	\item We develop a data-driven PI algorithm to solve the LQ mean-field game problem. By virtue of the proposed system transformation idea, a novel data-driven PI algorithm is proposed to circumvent the need of all system coefficients. The simulation results show that this algorithm successfully obtains an $\epsilon$-Nash equilibrium with errors similar to those of \cite[Section 4]{XuShenHuang2023}. But it can be noted that our algorithm may be more computationally efficient than their algorithm (see Remark \ref{remark1}). 
	\item We develop a data-driven VI algorithm to deal with the LQ mean-field game problem. It is worth mentioning that the PI algorithms require a priori knowledge of two Hurwitz matrices and these Hurwitz matrices are closely related to the system coefficients. When all system matrices are unavailable, it may be difficult to obtain two matrices that satisfy this condition. As a consequence, we develop a data-driven VI algorithm to  overcome this difficulty. The proposed data-driven VI algorithm neither requires the system parameter information nor the assumption of two initial Hurwitz matrices.
\end{itemize} 

The rest of this paper is organized as follows. In Section 2, we give some standard notations and terminologies, and formulate the continuous-time LQ mean-field game problem. In Section 3, we design the data-driven PI algorithm to work out the game problem. In Section 4, we develop the data-driven VI algorithm to solve the problem. In Section 5, we validate the obtained algorithms by means of two simulation examples. In Section 6, we give some concluding remarks and outlooks.

\section{Problem formulation and preliminaries}
\subsection{Notations}
Let us denote $\mathbb{Z}^+$ and $\mathbb{Z}$ the set of  positive integers and nonnegative integers. Let $\mathbb{R}$, $\mathbb{R}^{m\times n}$ and $\mathbb{R}^m$ be the collection of real numbers, $m\times n$-dimensional real matrices and $m$-dimensional real vectors, respectively. Denote the $m\times m$ identity matrix as $\mathbb{I}_m$. For notation simplicity, we denote any zero vector or zero matrix by $0$. Let $diag\{\nu\}$ be a diagonal matrix whose diagonal is vector $\nu$. $|\cdot|$ represents the induced matrix norm. Let $\mathbb{S}^m$, $\mathbb{S}^m_+$ and $\mathbb{S}^m_{++}$ denote the set of symmetric matrices, positive semidefinite matrices and positive definite matrices in $\mathbb{R}^{m\times m}$. A positive semidefinite (respectively, positive definite) matrix $M$ is denoted by $M\geq0$ (respectively, $M> 0$). Given $M\in\mathbb{R}^{m\times m}$, let $Re(\lambda_l(M))$, $l=1,2,\cdots,m$, denote the real part of the $l$-th eigenvalue of $M$. Let superscript $T$ be the transpose of any vector or matrix. Furthermore, $\otimes$ represents the Kronecker product. Given a matrix $M\in\mathbb{R}^{m\times n}$,  $M^\dagger$ denotes its pseudoinverse and $vec(M)\triangleq[M_1^T,M_2^T,\cdots,M_n^T]^T$, in which $M_l\in\mathbb{R}^m$, $l=1,2,\cdots,n$, denotes the $l$-th column of matrix $M$. Given a symmetric matrix $S\in\mathbb{S}^m$, we define $vecs(S)\triangleq[S_{11},2S_{12},2S_{13},\cdots,2S_{1m},S_{22},2S_{23},\cdots,2S_{m-1m},S_{mm}]^T$, where $S_{rl}$, $r,l=1,2,\cdots, m$, is the $(r,l)$-th element of matrix $S$. For a given vector $\zeta\in\mathbb{R}^m$, let $\overline{\zeta}\triangleq[\zeta_1^2,\zeta_1\zeta_2,\cdots,\zeta_1\zeta_m,\zeta_2^2,\zeta_2\zeta_3,\cdots,\zeta_{m-1}\zeta_m,\zeta_m^2]^T$, where $\zeta_l$, $l=1,2,\cdots,m$, denotes the $l$-th element of  $\zeta$. For any matrix $M\in\mathbb{R}^{m\times n}$ and vector $\varphi\in\mathbb{R}^n$, we denote by $[M]_{rl}$, $r=1,2,\cdots,m$, $l=1,2,\cdots,n$, the $(r,l)$-th element of $M$ and $[\varphi]_r$,  $r=1,2,\cdots,n$, the $r$-th element of $\phi$. 

\subsection{LQ mean-field game problems} 

We study a stochastic differential game with $N$ agents in a large-population framework. The state of agent $i$ satisfies
\begin{equation}\label{system}
	\left\{
	\begin{aligned}
		dx_i(t)&=\left[Ax_i(t)+Bu_i(t)\right]dt+CdW_i(t),\\
		x_i(0)&=x_{i0},\quad 1\leq i \leq N,
	\end{aligned}\right.
\end{equation} 
where  $x_i(\cdot)\in\mathbb{R}^n$ and $u_i(\cdot)\in\mathbb{R}^m$ represent the state and the control strategy of agent $i$, respectively.  $\left\{W_i(\cdot)\right\}_{i=1}^N$ are  independent standard $p$-dimensional Brownian motions defined on a filtered complete probability space $\left(\Omega,\mathcal{F},\left(\mathcal{F}_t\right)_{t\geq 0},\mathbb{P}\right)$, where $\mathcal F_t\triangleq\sigma\left(W_i(s): 1\leq i \leq N,\ 0\leq s\leq t\right)$. The system coefficients $A$, $B$ and $C$ are three unknown constant matrices of proper sizes.  The initial states $\left\{x_{i0}\right\}_{i=1}^N$ are independent
and identically distributed with the same expectation. It is assumed that the initial states are independent of $\left\{W_i(\cdot)\right\}_{i=1}^N$ and their second moments are finite. Moreover, let $\mathcal F^{W_i}_t\triangleq\sigma\left(W_i(s):0\leq s\leq t\right)$ be the information available to agent $i$, $1\leq i \leq N$.

The admissible control set of agent $i$ is 
\begin{equation*}
	\mathcal{U}_{ad,i}\triangleq\left\{u_i(\cdot):[0,+\infty)\times\Omega\rightarrow\mathbb{R}^m\left|u_i(\cdot)\ \text{is adapted to}\ \mathcal F^{W_i}_t,\ \mathbb{E}\int_0^{+\infty} e^{-\rho t}|u(t)|^2dt<+\infty\right.\right\},\ 1\leq i\leq N,
\end{equation*}
and all agents’ admissible control set is  $\mathcal{U}_{ad}=\mathcal{U}_{ad,1}\times\mathcal{U}_{ad,2}\times\cdots\times\mathcal{U}_{ad,N}$, where $\rho>0$ is a discount parameter.

The cost functional of agent $i$ is
\begin{equation*}
	\mathcal{J}_i(u_i(\cdot),u_{-i}(\cdot))\triangleq\mathbb{E}\int_{0}^{+\infty}e^{-\rho t} \left[\left(x_i(t)-x_N(t)\right)^TQ\left(x_i(t)-x_N(t)\right)+u_i(t)^TRu_i(t)\right]dt, \ 1\leq i\leq N,
\end{equation*}
where $Q>0$, $R>0$ are two weighting matrices,  $x_N(\cdot)\triangleq\frac{1}{N}\sum_{i=1}^Nx_i(\cdot)$ and $u_{-i}(\cdot)\triangleq\left(u_1(\cdot),\cdots,u_{i-1}(\cdot),u_{i+1}(\cdot),\cdots,u_N(\cdot)\right)$. 

Then the infinite-horizon LQ mean-field game problem is as follows.

\vspace{0.2cm}

\noindent \textbf{Problem (LQG)} Find $u^*(\cdot)=(u_1^*(\cdot),u_2^*(\cdot),\cdots,u_N^*(\cdot))\in\mathcal{U}_{ad}$ satisfying
\begin{equation*}
	\mathcal{J}_i(u^*_i(\cdot),u^*_{-i}(\cdot))=\min_{u_i(\cdot)\in\mathcal{U}_{ad,i}}\mathcal{J}_i(u_i(\cdot),u^*_{-i}(\cdot)), \ 1\leq i\leq N,
\end{equation*}
where $u^*_{-i}(\cdot)=\left(u^*_1(\cdot),\cdots,u^*_{i-1}(\cdot),u^*_{i+1}(\cdot),\cdots,u^*_N(\cdot)\right)$.

\vspace{0.2cm}

Similar to \cite{Huang 2010,XuShenHuang2023}, we are interested in designing decentralized control strategies and the $\epsilon$-Nash equilibrium of Problem (LQG), which are closely related to two AREs
\begin{equation}\label{ARE_P}
	\rho P=PA+A^TP-PBR^{-1}B^TP+Q,
\end{equation}
and
\begin{equation}\label{ARE_S}
	\rho Y=YA+A^TY-YBR^{-1}B^TY.
\end{equation}

In order to guarantee the existence of solutions to these two equations, the following assumption is required.
\begin{assum}\label{assu1}
	 $Re(\lambda_l(A))\neq \rho$, $l=1,2,\cdots,n$, and $(A, B)$ is stabilizable.
\end{assum}

By virtue of AREs (\ref{ARE_P}) and (\ref{ARE_S}), the next lemma presents the $\epsilon$-Nash equilibrium of Problem (LQG), whose proof follows from  \cite[Proposition 2.1, Remark 3.1]{XuShenHuang2023} and \cite[Theorem 11]{Huang 2010} and thus is omitted here.

\begin{lemma}\label{solution} Let Assumption \ref{assu1} hold. Then we have

(i) ARE (\ref{ARE_P}) admits a solution $P^*>0$ such that $A-BK^*$ is Hurwitz, and ARE (\ref{ARE_S}) admits a solution $Y^*\geq 0$ such that $A-BK^*_Y-0.5\rho\mathbb{I}_n$ is Hurwitz, where $K^*\triangleq R^{-1}B^TP^*$ and $K^*_Y\triangleq R^{-1}B^TY^*$; 
	
(ii) The decentralized strategies of Problem (LQG) are
\begin{equation*}
	u_i(t)=-K^*x_i(t)-(K_Y^*-K^*)\widehat{x}(t),\quad 1\leq i\leq N,
\end{equation*}
where $\widehat{x}(\cdot)$ is governed by the aggregate quantity
\begin{equation}\label{quantity}
	\begin{cases}
		d\widehat{x}(t)=\big(A-BK^*_Y\big)\widehat{x}(t)dt,\\
		\widehat{x}(0)=\xi_0.\\
	\end{cases}
\end{equation}
Moreover, $\big\{u_i(\cdot)\big\}_{i=1}^N$ construct an $\epsilon$-Nash equilibrium of Problem (LQG).
\end{lemma}

Lemma \ref{solution} illustrates that AREs (\ref{ARE_P}) and (\ref{ARE_S}) play a core role in constructing the $\epsilon$-Nash equilibrium. Noteworthily, in view of the nonlinear properties of (\ref{ARE_P}) and (\ref{ARE_S}), it is hard to directly solve the two equations.  Meanwhile, when all system coefficients are unavailable in the real world, it is more difficult to solve these two equations. In the next two sections, we will propose two algorithms to solve AREs (\ref{ARE_P}) and (\ref{ARE_S}) without knowing all system matrices. 

\section{A system transformation data-driven PI algorithm}
In this section, we devote ourselves to designing a novel data-driven PI algorithm to tackle Problem (LQG) without knowing all system coefficients.

Our algorithm is based on a model-based PI algorithm developed in \cite[Lemmas 3.1 and 3.2]{XuShenHuang2023}. We summarize the model-based PI algorithm in Algorithm \ref{model-based PI} and present its convergence results in the next lemma for future use.

\begin{lemma}\label{MPI}
	Suppose that Assumption \ref{assu1} holds. 
	Then the sequences $\{K^k\}_{k=1}^{+\infty}$ and $\{K^k_Y\}_{k=1}^{+\infty}$ obtained by Algorithm \ref{model-based PI} have the following properties:
		
		(i)  $\big\{A-0.5\rho\mathbb{I}_n-BK^k\big\}_{k=1}^{+\infty}$ and  $\big\{A-0.5\rho\mathbb{I}_n-BK^k_Y\big\}_{k=1}^{+\infty}$ are Hurwitz matrices;
		
		(ii) $\lim_{k\rightarrow+\infty}K^k=K^*$, $\lim_{k\rightarrow+\infty}K^k_Y=K^*_Y$.
\end{lemma}
\begin{algorithm}[h]
	\caption{Model-based PI algorithm}
	\label{model-based PI}
	\begin{algorithmic}[1]
		
		\State Initial $k=1$. Choose $K^0\in\mathbb{R}^{m\times n}$ and  $K^0_Y\in\mathbb{R}^{m\times n}$ such that $A-BK^0$ and $A-BK^0_Y$ are Hurwitz. Predefine a small  threshold $\varepsilon>0$. 
		
		\Loop
		
		\State Calculate $P^k$ and $Y^k$ by
		 \begin{subequations}
			\begin{equation}\label{PE1}
				\rho P^k=P^k(A-BK^{k-1})+(A-BK^{k-1})^TP^k+ (K^{k-1})^TRK^{k-1}+Q,	
			\end{equation}
			\begin{equation}\label{PE2}
				\rho Y^k=Y^k(A-BK^{k-1}_Y)+(A-BK^{k-1}_Y)^TY^k+ (K^{k-1}_Y)^TRK^{k-1}_Y.
			\end{equation}
		\end{subequations}
		
		\If{$|K^{k+1}-K^{k}|>\varepsilon$ or $|K^{k+1}_Y-K^{k}_Y|>\varepsilon$} 
		
		\State $K^k\leftarrow R^{-1}B^TP^k,\ K^k_Y\leftarrow R^{-1}B^TY^k.$
			
		 \State$k\leftarrow k+1.$
		
		\Else
		\State $\textbf{return} \,\,(P^k,\ K^k,\ Y^k,\ K^k_Y)$.
		\EndIf
		\EndLoop
	\end{algorithmic}
\end{algorithm}

It is worth pointing out that Algorithm \ref{model-based PI} needs all information of system matrices, which may be hard to obtain in the real world. In the sequel, we will remove the requirement of all system matrices. 

To this end, we first handle system (\ref{system}) of any given agent $i$. By defining $\mathcal{X}(\cdot)\triangleq\mathbb{E}[x_i(\cdot)]$, $\mathcal{V}(\cdot)\triangleq\mathbb{E}[u_i(\cdot)]$ and $\mathcal{X}_0\triangleq\mathbb{E}[x_{i0}]$, we transform system (\ref{system}) into 
\begin{equation}\label{system2}
	\begin{cases}
		d\mathcal{X}(t)=\big[A\mathcal{X}(t)+B\mathcal{V}(t)\big]dt,\\
		\mathcal{X}(0)=\mathcal{X}_0.
	\end{cases}
\end{equation}  

Next, we give a lemma to reveal a relationship between system (\ref{system2}) and the matrices to be solved in Algorithm \ref{model-based PI}, which will play a central role in deriving our data-driven PI algorithm.

\begin{lemma}\label{relationship} For any $k\in\mathbb{Z}^+$, $P^k$ generated by Algorithm \ref{model-based PI} satisfies
	\begin{equation}\label{relationship1}
		\begin{split}
			&\varphi(s_j,s_{j+1})vecs(P^k)
			-2\Big\{\int_{s_j}^{s_{j+1}}e^{-\rho t}\big[\mathcal{X}(t)^T\otimes\mathcal{X}(t)^T\big]dt\Big\}\big(\mathbb{I}_n\otimes (K^{k-1})^T\big)vec(\mathcal{L}^k)\\
			&-2\Big\{\int_{s_j}^{s_{j+1}}e^{-\rho t}\big[\mathcal{X}(t)^T\otimes\mathcal{V}(t)^T\big]dt\Big\}vec(\mathcal{L}^k)\\
			=\,\,&\Big\{\int_{s_j}^{s_{j+1}}e^{-\rho t}\big[\mathcal{X}(t)^T\otimes\mathcal{X}(t)^T\big]dt\Big\}vec\big(-(K^{k-1})^TRK^{k-1}-Q\big),\\
		\end{split}
	\end{equation}
and $Y^k$ generated by Algorithm \ref{model-based PI} satisfies
\begin{equation}\label{relationship2}
	\begin{split}
		&\varphi(s_j,s_{j+1})vecs(Y^k)
		-2\Big\{\int_{s_j}^{s_{j+1}}e^{-\rho t}\big[\mathcal{X}(t)^T\otimes\mathcal{X}(t)^T\big]dt\Big\}\big(\mathbb{I}_n\otimes (K^{k-1}_Y)^T\big)vec(\mathcal{L}^k_Y)\\
		&-2\Big\{\int_{s_j}^{s_{j+1}}e^{-\rho t}\big[\mathcal{X}(t)^T\otimes\mathcal{V}(t)^T\big]dt\Big\}vec(\mathcal{L}^k_Y)\\
		=\,\,&\Big\{\int_{s_j}^{s_{j+1}}e^{-\rho t}\big[\mathcal{X}(t)^T\otimes\mathcal{X}(t)^T\big]dt\Big\}vec\big(-(K^{k-1}_Y)^TRK^{k-1}_Y\big),\\
	\end{split}
\end{equation}
where $\varphi(s_j,s_{j+1})\triangleq e^{-\rho s_{j+1}}\overline{\mathcal{X}}(s_{j+1})^T-e^{-\rho s_j}\overline{\mathcal{X}}(s_j)^T$, $\mathcal{L}^k\triangleq  B^TP^k$, $\mathcal{L}^k_Y\triangleq B^TY^k$, $\{s_j\}_{j=0}^d$ is a set of real numbers satisfying $s_0<s_1<s_2<\cdots<s_d$ and $d\in\mathbb{Z}^+$ is a predefined positive integer.
\end{lemma}

\begin{proof} Since the derivations of (\ref{relationship1}) and (\ref{relationship2}) are similar, we only provide the proof of (\ref{relationship1}). According to system (\ref{system2}), we know
	\begin{equation}\label{eq8}
		\begin{split}
			&d\big(e^{-\rho t}\mathcal{X}(t)^TP^k\mathcal{X}(t)\big)\\
			=\,\,&\Big\{-\rho e^{-\rho t}\mathcal{X}(t)^TP^k\mathcal{X}(t)+e^{-\rho t}\big[A\mathcal{X}(t)+B\mathcal{V}(t)\big]^TP^k\mathcal{X}(t)+e^{-\rho t}\mathcal{X}(t)^TP^k\big[A\mathcal{X}(t)+B\mathcal{V}(t)\big]\Big\}dt\\
			=\,\,&e^{-\rho t}\Big\{\mathcal{X}(t)^T\big
			[-\rho P^k+A^TP^k+P^kA\big]\mathcal{X}(t)+2\mathcal{V}(t)^TB^TP^k\mathcal{X}(t)\Big\}dt.
		\end{split}
	\end{equation}
With the help of (\ref{PE1}), it is clear that
\begin{equation}\label{eq9}
	\begin{split}
		-\rho P^k+A^TP^k+P^kA=P^kBK^{k-1}+(K^{k-1})^TB^TP^k-(K^{k-1})^TRK^{k-1}-Q.	
	\end{split}
\end{equation}
Inserting (\ref{eq9}) into (\ref{eq8}), (\ref{eq8}) is transformed into
 	\begin{equation}\label{eq10}
 	\begin{split}
 		&d\big(e^{-\rho t}\mathcal{X}(t)^TP^k\mathcal{X}(t)\big)\\
 		=\,\,&e^{-\rho t}\Big\{\mathcal{X}(t)^T\big
 		[2(K^{k-1})^TB^TP^k-(K^{k-1})^TRK^{k-1}-Q	\big]\mathcal{X}(t)+2\mathcal{V}(t)^TB^TP^k\mathcal{X}(t)\Big\}dt.
 	\end{split}
 \end{equation}
Then, integrating the above equation from $s_j$ to $s_{j+1}$,  (\ref{eq10}) becomes
 \begin{equation*}
  	\begin{split}
  		&e^{-\rho s_{j+1}}\mathcal{X}(s_{j+1})^TP^k\mathcal{X}(s_{j+1})-e^{-\rho s_j}\mathcal{X}(s_j)^TP^k\mathcal{X}(s_j)\\
  		=\,\,&\int_{s_j}^{s_{j+1}}e^{-\rho t}\Big\{\mathcal{X}(t)^T\big
  		[2(K^{k-1})^TB^TP^k-(K^{k-1})^TRK^{k-1}-Q	\big]\mathcal{X}(t)+2\mathcal{V}(t)^TB^TP^k\mathcal{X}(t)\Big\}dt.
  	\end{split}
  \end{equation*}
 Thus, it follows from Kronecker product theory that 
 \begin{equation*}
 	\begin{split}
 		&\Big[e^{-\rho s_{j+1}}\overline{\mathcal{X}}(s_{j+1})^T-e^{-\rho s_j}\overline{\mathcal{X}}(s_j)^T\Big]vecs(P^k)\\
 		=\,\,&\Big\{\int_{s_j}^{s_{j+1}}e^{-\rho t}\big[\mathcal{X}(t)^T\otimes\mathcal{X}(t)^T\big]dt\Big\}vec\big(-(K^{k-1})^TRK^{k-1}-Q\big)\\
 		&+2\Big\{\int_{s_j}^{s_{j+1}}e^{-\rho t}\big[\mathcal{X}(t)^T\otimes\mathcal{X}(t)^T\big]dt\Big\}\big(\mathbb{I}_n\otimes (K^{k-1})^T\big)vec(B^TP^k)\\
 		&+2\Big\{\int_{s_j}^{s_{j+1}}e^{-\rho t}\big[\mathcal{X}(t)^T\otimes\mathcal{V}(t)^T\big]dt\Big\}vec(B^TP^k),\\
 	\end{split}
 \end{equation*}
which yields equation (\ref{relationship1}). This completes the proof.
\end{proof}

To use the lemma proposed above, we define 
\begin{equation}\label{symbol}
	\begin{split}
		\mathcal{I}  &\triangleq\bigg[\varphi(s_0,s_{1})^T,\varphi(s_1,s_2)^T,\cdots,\varphi(s_{d-1},s_d)^T\bigg]^T,\\	
		\mathcal{I}  _{\mathcal{X}}&\triangleq\bigg[\int_{s_{0}}^{s_{1}}e^{-\rho t}\big[\mathcal{X}(t)\otimes\mathcal{X}(t)\big]dt,\int_{s_{1}}^{s_{2}}e^{-\rho t}\big[\mathcal{X}(t)\otimes\mathcal{X}(t)\big]dt,\cdots,
		\int_{s_{d-1}}^{s_{d}}e^{-\rho t}\big[\mathcal{X}(t)\otimes\mathcal{X}(t)\big]dt\bigg]^T,\\
		\mathcal{I} _{\mathcal{X\mathcal{V}}}&\triangleq\bigg[\int_{s_{0}}^{s_{1}}e^{-\rho t}\big[\mathcal{X}(t)\otimes\mathcal{V}(t)\big]dt,\int_{s_{1}}^{s_{2}}e^{-\rho t}\big[\mathcal{X}(t)\otimes\mathcal{V}(t)\big]dt,\cdots,
		\int_{s_{d-1}}^{s_{d}}e^{-\rho t}\big[\mathcal{X}(t)\otimes\mathcal{V}(t)\big]dt\bigg]^T.\\
	\end{split}
\end{equation}
Thanks to the matrices in (\ref{symbol}), equations (\ref{relationship1}) and (\ref{relationship2}) imply that 
\begin{equation}\label{solve1}
	\Delta^k\begin{bmatrix}
		vecs(P^k)\\
		vec(\mathcal{L}^k)\\
	\end{bmatrix}=\Theta^k, \quad\forall k\in\mathbb{Z}^+,
\end{equation}
\begin{equation}\label{solve2}
	\Delta^k_Y\begin{bmatrix}
		vecs(Y^k)\\
		vec(\mathcal{L}^k_Y)\\
	\end{bmatrix}=\Theta^k_Y, \quad\forall k\in\mathbb{Z}^+,
\end{equation}
where $\Phi_k$ and $\Xi_k $ are defined by
\begin{equation*}
	\begin{split}
		\Delta^k&\triangleq\big[\mathcal{I},-2\mathcal{I}  _{\mathcal{X}}\big(\mathbb{I}_n\otimes(K^{k-1})^T\big)-2\mathcal{I} _{\mathcal{X\mathcal{V}}}\big],
		\Theta^k\triangleq\mathcal{I}  _{\mathcal{X}}vec\big(-(K^{k-1})^TRK^{k-1}-Q\big),\\
		\Delta^k_Y&\triangleq\big[\mathcal{I},-2\mathcal{I}  _{\mathcal{X}}\big(\mathbb{I}_n\otimes(K^{k-1}_Y)^T\big)-2\mathcal{I} _{\mathcal{X\mathcal{V}}}\big],
		\Theta^k_Y\triangleq\mathcal{I}  _{\mathcal{X}}vec\big(-(K^{k-1}_Y)^TRK^{k-1}_Y\big).
	\end{split}
\end{equation*}

By virtue of the above analysis, we present our data-driven PI algorithm in Algorithm \ref{data-driven PI} and provide its convergence results in Theorem \ref{convergence PI}.
\begin{algorithm}[h]
	\caption{System transformation data-driven PI algorithm}
	\label{data-driven PI}
	\begin{algorithmic}[1]
		
		\State Initial $k=1$, select a positive integer $d$ and a series of real numeber $\{s_j\}_{j=0}^d$. Choose $K^0\in\mathbb{R}^{m\times n}$ and  $K^0_Y\in\mathbb{R}^{m\times n}$ such that $A-BK^0$ and $A-BK^0_Y$ are Hurwitz. Predefine a small  threshold $\varepsilon>0$. Compute $\mathcal{I}$, $\mathcal{I}  _{\mathcal{X}}$ and $\mathcal{I} _{\mathcal{X\mathcal{V}}}$.
		
		\Loop
		
		\State Calculate ($P^k$, $\mathcal{L}^k$) and ($Y^k$, $\mathcal{L}^k_Y$), respectively, by
		\begin{equation}\label{solve3}
			\begin{bmatrix}
				vecs(P^k)\\
				vec(\mathcal{L}^k)\\
			\end{bmatrix}=(\Delta^k)^\dagger\Theta^k, 
		\end{equation}
	\begin{equation}\label{solve4}
		\begin{bmatrix}
			vecs(Y^k)\\
			vec(\mathcal{L}^k_Y)\\
		\end{bmatrix}=(\Delta^k_Y)^\dagger\Theta^k_Y.
	\end{equation}
		
		\If{$|K^{k+1}-K^{k}|>\varepsilon$ or $|K^{k+1}_Y-K^{k}_Y|>\varepsilon$} 
		
		\State  $K^{k+1}\leftarrow R^{-1}\mathcal{L}^{k}$, $K^{k+1}_Y\leftarrow R^{-1}\mathcal{L}^{k}_Y$.
		
		\State $k\leftarrow k+1$.
		
		\Else
		\State $\textbf{return} \,\,(P^k,\ K^k,\ Y^k,\ K^k_Y)$.
		\EndIf
		\EndLoop
	\end{algorithmic}
\end{algorithm}

\begin{theorem}\label{convergence PI}
	Suppose Assumption 1 holds and there exists a positive integer $\hat{d}$ such that 
	\begin{equation}\label{rank}
		rank\big([\mathcal{I}_{\mathcal{X}},\mathcal{I} _{\mathcal{X\mathcal{V}}}]\big)=mn+\frac{n(n+1)}{2}
	\end{equation}
holds for any $d\geq \hat{d}$, then $\{K^k\}_{k=1}^{+\infty}$ and $\{K^k_Y\}_{k=1}^{+\infty}$ generated by Algorithm \ref{data-driven PI} satisfy $\lim_{k\rightarrow+\infty}K^k=K^*$ and $\lim_{k\rightarrow+\infty}K^k_Y=K^*_Y$, respectively. 
\end{theorem}

\begin{proof}
	\textbf{Step 1:} Given $k\in\mathbb{Z}^+$, we first show that matrices $\Delta^k$ and $\Delta^k_Y$ have full column rank under condition (\ref{rank}). 
	
	We prove this result by contradiction. If the column ranks of $\Delta^k$ and $\Delta^k_Y$ are not full, then there exist two nonzero vectors $g\in\mathbb{R}^{mn+n(n+1)/2}$ and $h\in\mathbb{R}^{mn+n(n+1)/2}$ such that $\Delta^kg= 0$ and $\Delta^k_Yh=0$. Moreover, we can take four matrices $G_1\in\mathbb{S}^n$, $G_2\in\mathbb{R}^{m\times n}$, $H_1\in\mathbb{S}^n$ and $H_2\in\mathbb{R}^{m\times n}$ by $[vecs(G_1)^T,vec(G_2)^T]=g$ and $[vecs(H_1)^T,vec(H_2)^T]=h$. 
	
	Similar to (\ref{eq8}), we have
	\begin{equation}\label{eq19}
		\begin{split}
			&e^{-\rho s_{j+1}}\mathcal{X}(s_{j+1})^TG_1\mathcal{X}(s_{j+1})-e^{-\rho s_j}\mathcal{X}(s_j)^TG_1\mathcal{X}(s_j)\\
			=&\int_{s_j}^{s_{j+1}}e^{-\rho t}\Big\{\mathcal{X}(t)^T\big
			[-\rho G_1+A^TG_1+G_1A\big]\mathcal{X}(t)+2\mathcal{V}(t)^TB^TG_1\mathcal{X}(t)\Big\}dt,\\
		\end{split}
	\end{equation}
	and
	\begin{equation}\label{eq20}
		\begin{split}
			&e^{-\rho s_{j+1}}\mathcal{X}(s_{j+1})^TH_1\mathcal{X}(s_{j+1})-e^{-\rho s_j}\mathcal{X}(s_j)^TH_1\mathcal{X}(s_j)\\
			=&\int_{s_j}^{s_{j+1}}e^{-\rho t}\Big\{\mathcal{X}(t)^T\big
			[-\rho H_1+A^TH_1+H_1A\big]\mathcal{X}(t)+2\mathcal{V}(t)^TB^TH_1\mathcal{X}(t)\Big\}dt.\\
		\end{split}
	\end{equation}
	Following the derivations of (\ref{solve1}) and (\ref{solve2}) and combining (\ref{eq19}) and (\ref{eq20}), we derive 
	\begin{equation}\label{eq21}
		\begin{split}
			0&=\Delta^kg=\Delta^k\begin{bmatrix}
				vecs(G_1)\\
				vec(G_2)\\
			\end{bmatrix}=[\mathcal{I}_{\mathcal{X}},\mathcal{I} _{\mathcal{X\mathcal{V}}}]\begin{bmatrix}
				vec(\mathbb{G}_1)\\
				vec(\mathbb{G}_2)\\
			\end{bmatrix},\\
			0&=\Delta^k_Yh=\Delta^k_Y\begin{bmatrix}
				vecs(H_1)\\
				vec(H_2)\\
			\end{bmatrix}=[\mathcal{I}_{\mathcal{X}},\mathcal{I} _{\mathcal{X\mathcal{V}}}]\begin{bmatrix}
				vec(\mathbb{H}_1)\\
				vec(\mathbb{H}_2)\\
			\end{bmatrix},
		\end{split}
	\end{equation}
	where
	\begin{equation}\label{eq23}
		\begin{split}
			\mathbb{G}_1&=-\rho G_1+A^TG_1+G_1A-(K^{k-1})^TG_2-G_2^TK^{k-1},\\
			\mathbb{G}_2&=2B^TG_1-2G_2,\\
			\mathbb{H}_1&=-\rho H_1+A^TH_1+H_1A-(K^{k-1}_Y)^TH_2-H_2^TK^{k-1}_Y,\\
			\mathbb{H}_2&=2B^TH_1-2H_2.\\
		\end{split}
	\end{equation}
	
	Noting that $\mathbb{G}_1$ and $\mathbb{H}_1$ are symmetric matrices,  it follows that 
	\begin{equation}\label{eq24}
		\mathcal{I}_{\mathcal{X}}vec(\mathbb{G}_1)=\widehat{\mathcal{I}}_{\mathcal{X}}vecs(\mathbb{G}_1),\quad \mathcal{I}_{\mathcal{X}}vec(\mathbb{H}_1)=\widehat{\mathcal{I}}_{\mathcal{X}}vecs(\mathbb{H}_1),
	\end{equation}
	where 
	\begin{equation}\label{symbolI}
		\widehat{\mathcal{I}}_{\mathcal{X}}\triangleq\bigg[\int_{s_{0}}^{s_{1}}e^{-\rho t}\overline{\mathcal{X}}(t)dt,\int_{s_{1}}^{s_{2}}e^{-\rho t}\overline{\mathcal{X}}(t)dt,\cdots,
		\int_{s_{d-1}}^{s_{d}}e^{-\rho t}\overline{\mathcal{X}}(t)dt\bigg]^T.
	\end{equation}
	Keeping (\ref{eq24}) in mind, (\ref{eq21}) implies
	\begin{equation}\label{eq25}
		\begin{split}
			0&=\Delta^kg=[\mathcal{I}_{\mathcal{X}},\mathcal{I} _{\mathcal{X\mathcal{V}}}]\begin{bmatrix}
				vec(\mathbb{G}_1)\\
				vec(\mathbb{G}_2)\\
			\end{bmatrix}=[\widehat{\mathcal{I}}_{\mathcal{X}},\mathcal{I} _{\mathcal{X\mathcal{V}}}]\begin{bmatrix}
				vecs(\mathbb{G}_1)\\
				vec(\mathbb{G}_2)\\
			\end{bmatrix},\\
			0&=\Delta^k_Yh=[\mathcal{I}_{\mathcal{X}},\mathcal{I} _{\mathcal{X\mathcal{V}}}]\begin{bmatrix}
				vec(\mathbb{H}_1)\\
				vec(\mathbb{H}_2)\\
			\end{bmatrix}=[\widehat{\mathcal{I}}_{\mathcal{X}},\mathcal{I} _{\mathcal{X\mathcal{V}}}]\begin{bmatrix}
				vecs(\mathbb{H}_1)\\
				vec(\mathbb{H}_2)\\
			\end{bmatrix}.
		\end{split}
	\end{equation}
	Under condition (\ref{rank}), it is evident to see that matrix $[\widehat{\mathcal{I}}_{\mathcal{X}},\mathcal{I} _{\mathcal{X\mathcal{V}}}]$ has full column rank. Combining this fact with (\ref{eq25}), the definitions of $vec(\cdot)$ and $vecs(\cdot)$ mean that $\mathbb{G}_1=\mathbb{G}_2=\mathbb{H}_1=\mathbb{H}_2=0$. Thus, (\ref{eq23}) yields  
	\begin{equation}\label{eq27}
		\begin{split}
			(A-0.5\rho\mathbb{I}_n-BK^{k-1})^TG_1+G_1(A-0.5\rho-B^TK^{k-1})&=0,\\
			(A-0.5\rho\mathbb{I}_n-BK^{k-1}_Y)^TH_1+H_1(A-0.5\rho-B^TK^{k-1}_Y)&=0.
		\end{split}
	\end{equation}
	According to the Kronecker product theory, (\ref{eq27}) shows 
	\begin{equation}\label{eq28}
		\begin{split}
			\Big[\mathbb{I}_n\otimes(A-0.5\rho\mathbb{I}_n-BK^{k-1})^T+(A-0.5\rho-B^TK^{k-1})^T\otimes\mathbb{I}_n\Big]vec(G_1)&=0,\\
			\Big[\mathbb{I}_n\otimes(A-0.5\rho\mathbb{I}_n-BK^{k-1}_Y)^T+(A-0.5\rho-B^TK^{k-1}_Y)^T\otimes\mathbb{I}_n\Big]vec(H_1)&=0.\\
		\end{split}
	\end{equation}
	
	By virtue of (i) of Lemma \ref{MPI}, we know matrices $\mathbb{I}_n\otimes(A-0.5\rho\mathbb{I}_n-BK^{k-1})^T+(A-0.5\rho-B^TK^{k-1})^T\otimes\mathbb{I}_n$ and $\mathbb{I}_n\otimes(A-0.5\rho\mathbb{I}_n-BK^{k-1}_Y)^T+(A-0.5\rho-B^TK^{k-1}_Y)^T\otimes\mathbb{I}_n$ are Hurwitz. Then it follows from (\ref{eq28}) that $G_1=H_1=0$. Combing it with $\mathbb{G}_2=\mathbb{H}_2=0$, we obtain $G_1=H_1=G_2=H_2=0$, which means that $g=h=0$. However, this is contrary to the assumptions that $g$ and $h$ are nonzero vectors.
	
	\textbf{Step 2:} We now  prove the convergence results of Algorithm \ref{data-driven PI}. For any $k\in\mathbb{Z}^+$, since $\Delta^k$ and $\Delta^k_Y$ have full column rank, (\ref{solve3}) and (\ref{solve4}) admit a unique solution, respectively. Furthermore, Lemma \ref{relationship} shows that $(P^k, \mathcal{L}^k)$ is a solution of (\ref{solve3}) and $(Y^k, \mathcal{L}^k_Y)$ is a solution of (\ref{solve4}). Therefore, (\ref{solve3}) admits unique solution $(P^k, \mathcal{L}^k)$ and (\ref{solve4}) admits unique solution $(Y^k, \mathcal{L}^k_Y)$. Thus, the convergence of Algorithm \ref{data-driven PI} follows from Lemma \ref{MPI}. We have then proved the theorem.
\end{proof}

\begin{rem}\label{remark1}
	To further analyze our algorithm, we are now in a position to consider the computational complexity (i.e. time complexity and space complexity ) of Algorithm \ref{data-driven PI}. If each integral in (\ref{symbol}) is computed by summation with $y$ discrete points and the samples of employing Monte Carlo is $M$, then the time complexity of calculating matrices $\mathcal{I}$, $\mathcal{I}_{\mathcal{X}}$ and $\mathcal{I} _{\mathcal{X\mathcal{V}}}$ is $O(M+dy)$. In this case, the time complexity of computing the similar matrices in Xu et al. \cite{XuShenHuang2023} is $O(Mdy)$. Moreover, it is easy to verify that the space complexity of Algorithm \ref{data-driven PI} is the same as the algorithm in Xu et al. \cite{XuShenHuang2023}. Due to the fact that $M$ is always large in practice, it is evident that the computational complexity of our algorithm is much smaller than that of the algorithm in Xu et al. \cite{XuShenHuang2023}. 
\end{rem}

\begin{rem}
	In practical implementations, rank condition (\ref{rank}) can be met by adding an exploration signal to the input. There are many types of exploration signals commonly used in RL algorithms, such as Gaussian signals \cite{Bradtke1993,Zhang2023,ZhangLi2023}, random signals \cite{Tamimi2007,XuLewis2012}  and signals generated by the sum of sinusoidal functions \cite{JiangJiangPI2012,XuShenHuang2023}. In the simulation sections of this paper, we adopt exploration signals mainly generated by the combination of sinusoidal functions.  
\end{rem}

	By virtue of a system transformation idea, we have developed a data-driven PI algorithm to solve Problem (LQG) without needing all system coefficients. However, it is evident that this algorithm requires two matrices $K^0\in\mathbb{R}^{m\times n}$ and  $K^0_Y\in\mathbb{R}^{m\times n}$ so that $A-BK^0$ and $A-BK^0_Y$ are Hurwitz. When all system matrices are unknown, it may be difficult to obtain two matrices that meet this condition. Thus, we will develop a data-driven VI algorithm in the next section to solve this conundrum.

\section{A system transformation data-driven VI algorithm}
In this section, we aim to propose a data-driven VI algorithm to remove the assumption of needing two Hurwitz matrices. Moreover, the data-driven VI algorithm also does not require the knowledge of all system coefficients. 

In addition to Assumption \ref{assu1}, this section also needs the following assumption. 
\begin{assum}\label{assu2}
	$A-0.5\rho\mathbb{I}_n$ is Hurwitz.
\end{assum}

Under Assumptions \ref{assu1}-\ref{assu2}, it is easy to verify that $Y^*=0$ and $K^*_Y=0$. Thus, we only need to deal with ARE (\ref{ARE_P}) and the corresponding $K^*$. To proceed, we define a constant sequence $\{\gamma_k\}_{k=1}^{+\infty}$  and a set of bounded collections $\{\mathcal{D}_q\}_{q=0}^{+\infty}$, which satisfy
\begin{equation}\label{se1}
 \sum_{k=0}^{+\infty}\gamma_k=+\infty,\
	\sum_{k=0}^{+\infty}\gamma_k^2<+\infty,\
		\gamma_k>0, \ \forall k\in\mathbb{Z},
\end{equation}
and 
\begin{equation}\label{se2}
 \lim\limits_{q\rightarrow+\infty}\mathcal{D}_q=\mathbb{S}^n_+,\
 	\mathcal{D}_q\subseteq \mathcal{D}_{q+1},\ \forall q\in\mathbb{Z}.
\end{equation}

Based on the above symbols, we propose a model-based VI algorithm in Algorithm \ref{model-based VI}, which can initiate from any positive definite matrix. Since the model-based VI algorithm is a direct extension of \cite[Lemma 3.4, Theorem 3.3]{BianJiangVI2016}, we present its convergence results in the next lemma but omit its proof due to space limitations

\begin{lemma}\label{MVI}
	Let Assumptions \ref{assu1}-\ref{assu2} hold. Then  $\{\mathcal{K}^k\}_{k=0}^{+\infty}$ generated by Algorithm \ref{model-based VI} satisfy $\lim_{k\rightarrow+\infty}\mathcal{K}^k=K^*$.
\end{lemma}
\begin{algorithm}[h]
	\caption{Model-based VI algorithm}
	\label{model-based VI}
	\begin{algorithmic}[1]
		
		\State Initial $k=0$ and $q=0$. Choose $\mathcal{P}^0>0$.  Choose a  sequence $\{\gamma_k\}_{k=1}^{+\infty}$ that meets (\ref{se1})  and a set of bounded collections $\{\mathcal{D}_q\}_{q=0}^{+\infty}$ that meets (\ref{se2}). Predefine a small  threshold $\varepsilon>0$.
		
		\Loop
		
		\State	$\mathcal{K}^k\leftarrow R^{-1}B^T\mathcal{P}^k$,\ 	$\widetilde{\mathcal{P}}\leftarrow \mathcal{P}^k+\gamma_k\big(\mathcal{P}^kA+A^T\mathcal{P}^k-\rho \mathcal{P}^k+(\mathcal{K}^k)^TR\mathcal{K}^k+Q\big).$

		\If{$|\widetilde{\mathcal{P}}-P^{k}|/\gamma_k<\varepsilon$} 
		
		\State $\textbf{return} \,\,(\mathcal{P}^k,\ \mathcal{K}^k)$.
		
		\ElsIf{$\widetilde{\mathcal{P}}\notin \mathcal{D}_q$}
		
		\State $\mathcal{P}^{k+1}\leftarrow \mathcal{P}^0, \ q\leftarrow q+1.$
		
		\Else
		\State $\mathcal{P}^{k+1}\leftarrow \widetilde{\mathcal{P}}.$
		\EndIf
		\State$k\leftarrow k+1.$
		\EndLoop
	\end{algorithmic}
\end{algorithm}

Although Algorithm \ref{model-based VI} does not need the assumption of two Hurwitz matrices, it still requires the information of system coefficients $A$ and $B$. In the sequel, we will develop a data-driven VI algorithm to solve Problem (LQG) without depending on all system matrices. To this end, we give a lemma to construct a relationship between system (\ref{system2}) and the matrices in Algorithm \ref{model-based VI}. 

\begin{lemma}\label{relationshipVI} For any $k\in\mathbb{Z}$, $\mathcal{P}^k$ generated by Algorithm \ref{model-based VI} satisfies
	\begin{equation}\label{relationshipVI1}
		\begin{split}
			&\Big[e^{-\rho s_{j+1}}\overline{\mathcal{X}}(s_{j+1})^T-e^{-\rho s_j}\overline{\mathcal{X}}(s_j)^T\Big]vecs(\mathcal{P}^k)\\
			=\,\,&\Big\{\int_{s_j}^{s_{j+1}}e^{-\rho t}\overline{\mathcal{X}}(t)^Tdt\Big\}vecs\big(\mathcal{M}^k\big)
			+2\Big\{\int_{s_j}^{s_{j+1}}e^{-\rho t}\big[\mathcal{X}(t)^T\otimes\mathcal{V}(t)^T\big]dt\Big\}vec(\mathcal{N}^k),\\
		\end{split}
	\end{equation}
	where $\mathcal{M}^k\triangleq-\rho \mathcal{P}^k+A^T\mathcal{P}^k+\mathcal{P}^kA$, $\mathcal{N}^k\triangleq  B^T\mathcal{P}^k$, $\{s_j\}_{j=0}^d$ is a set of real numbers satisfying $s_0<s_1<s_2<\cdots<s_d$ and $d\in\mathbb{Z}^+$ is a predefined positive integer.
\end{lemma}

\begin{proof} By virtue of system (\ref{system2}), we derive
	\begin{equation*}
		\begin{split}
			&d\big(e^{-\rho t}\mathcal{X}(t)^T\mathcal{P}^k\mathcal{X}(t)\big)\\
			=\,\,&\Big\{-\rho e^{-\rho t}\mathcal{X}(t)^T\mathcal{P}^k\mathcal{X}(t)+e^{-\rho t}\big[A\mathcal{X}(t)+B\mathcal{V}(t)\big]^T\mathcal{P}^k\mathcal{X}(t)+e^{-\rho t}\mathcal{X}(t)^T\mathcal{P}^k\big[A\mathcal{X}(t)+B\mathcal{V}(t)\big]\Big\}dt\\
			=\,\,&e^{-\rho t}\Big\{\mathcal{X}(t)^T\big
			[-\rho \mathcal{P}^k+A^T\mathcal{P}^k+\mathcal{P}^kA\big]\mathcal{X}(t)+2\mathcal{V}(t)^TB^T\mathcal{P}^k\mathcal{X}(t)\Big\}dt.
		\end{split}
	\end{equation*}
	Then, integrating the above equation from $s_j$ to $s_{j+1}$,  we have
	\begin{equation*}
		\begin{split}
			&e^{-\rho s_{j+1}}\mathcal{X}(s_{j+1})^T\mathcal{P}^k\mathcal{X}(s_{j+1})-e^{-\rho s_j}\mathcal{X}(s_j)^T\mathcal{P}^k\mathcal{X}(s_j)\\
			=\,\,&\int_{s_j}^{s_{j+1}}e^{-\rho t}\Big\{\mathcal{X}(t)^T\big
			[-\rho \mathcal{P}^k+A^T\mathcal{P}^k+\mathcal{P}^kA\big]\mathcal{X}(t)+2\mathcal{V}(t)^TB^T\mathcal{P}^k\mathcal{X}(t)\Big\}dt.
		\end{split}
	\end{equation*}
According to Kronecker product theory,  we get (\ref{relationshipVI1}). This completes the proof.
\end{proof}

In light of the matrices in (\ref{symbol}) and (\ref{symbolI}), equation  (\ref{relationshipVI1}) means that
\begin{equation*}\label{solveVI1}
	\big[\widehat{\mathcal{I}}_{\mathcal{X}},2\mathcal{I}  _{\mathcal{X\mathcal{V}}}\big]\begin{bmatrix}
		vecs(\mathcal{M}^k)\\
		vec(\mathcal{N}^k)\\
	\end{bmatrix}=\mathcal{I}vecs(\mathcal{P}^k), \quad\forall k\in\mathbb{Z}.
\end{equation*}

Now we can summarize our data-driven VI algorithm in Algorithm \ref{data-driven VI}. It should be noted that this algorithm does not require  the assumption of Hurwitz matrices and can be implemented in the setting of unknown system matrices.

\begin{algorithm}[h]
	\caption{System transformation data-driven VI algorithm}
	\label{data-driven VI}
	\begin{algorithmic}[1]
		
		\State Initial $k=0$ and $q=0$. Choose $\mathcal{P}^0>0$.  Choose a  sequence $\{\gamma_k\}_{k=1}^{+\infty}$ that meets (\ref{se1})  and a set of bounded collections $\{\mathcal{D}_q\}_{q=0}^{+\infty}$ that meets (\ref{se2}). Predefine a small  threshold $\varepsilon>0$. Compute $\mathcal{I}$, $\widehat{\mathcal{I}}_{\mathcal{X}}$ and $\mathcal{I} _{\mathcal{X\mathcal{V}}}$.
		
		\Loop
		\State Calculate ($\mathcal{M}^k$, $\mathcal{N}^k$) by
		\begin{equation}\label{solveVI11}
		\begin{bmatrix}
				vecs(\mathcal{M}^k)\\
				vec(\mathcal{N}^k)\\
			\end{bmatrix}=	\big[\widehat{\mathcal{I}}_{\mathcal{X}},2\mathcal{I}  _{\mathcal{X\mathcal{V}}}\big]^\dagger\mathcal{I}vecs(\mathcal{P}^k). 
		\end{equation}

		\State	$\mathcal{K}^k\leftarrow R^{-1}\mathcal{N}^k$, 	$\widetilde{\mathcal{P}}\leftarrow \mathcal{P}^k+\gamma_k\Big(\mathcal{M}^k+(\mathcal{K}^k)^TR\mathcal{K}^k+Q\Big).$

		\If{$|\widetilde{\mathcal{P}}-P^{k}|/\gamma_k<\varepsilon$} 
		
		\State $\textbf{return} \,\,(\mathcal{P}^k, \mathcal{K}^k)$.
		
		\ElsIf{$\widetilde{\mathcal{P}}\notin \mathcal{D}_q$}
		
		\State $\mathcal{P}^{k+1}\leftarrow \mathcal{P}^0,\  q\leftarrow q+1.$
		
		\Else
		\State $\mathcal{P}^{k+1}\leftarrow \widetilde{\mathcal{P}}$.
		\EndIf
		\State$k\leftarrow k+1.$
		\EndLoop
	\end{algorithmic}
\end{algorithm}

\begin{theorem}
	Suppose Assumptions \ref{assu1}-\ref{assu2} hold and there exists a positive integer $\hat{d}$ such that 
	\begin{equation}\label{rank22}
		rank\big([\widehat{\mathcal{I}}_{\mathcal{X}},\mathcal{I} _{\mathcal{X\mathcal{V}}}]\big)=mn+\frac{n(n+1)}{2}
	\end{equation}
	holds for any $d\geq \hat{d}$, then $\{\mathcal{K}^k\}_{k=0}^{+\infty}$ obtained by Algorithm \ref{data-driven VI} satisfy $\lim_{k\rightarrow+\infty}\mathcal{K}^k=K^*$.
\end{theorem}

\begin{proof}
	Given $k\in\mathbb{Z}$, it follows from Lemma \ref{relationshipVI} that $(\mathcal{M}^k, \mathcal{N}^k)$ is a solution of (\ref{solveVI11}). When rank condition (\ref{rank22}) is guaranteed, we know (\ref{solveVI11}) admits a unique solution. Therefore, the solution matrices of Algorithm \ref{data-driven VI} are the same as those of Algorithm \ref{model-based VI}. Thus, Lemma \ref{MVI} implies the convergence of Algorithm \ref{data-driven VI}. This completes the proof.
\end{proof}

\section{Simulations}

This section demonstrates the applicability of the proposed data-driven algorithms through two numerical examples. 
\subsection{The first numerical example}\label{section5.1}
In this example, we let the coefficients  of system (\ref{system}) be the same as in \cite[Section 4]{XuShenHuang2023}, i.e.,
\begin{equation*}
	A=\begin{bmatrix}
		5 & 3\\
		10&12\\
	\end{bmatrix},
B=\begin{bmatrix}
	0\\
	1\\
\end{bmatrix},
	C=\begin{bmatrix}
	0.1 & 0.1\\
	0.1&0.1\\
\end{bmatrix}.
\end{equation*}

 The parameters in the performance index are $Q=10\mathbb{I}_n$, $R=1$ and $\rho=0.01$. In order to know the true values of $K^*$ and $K^*_Y$, we first carry out the model-based algorithm in Lemma \ref{MPI} with a large number of iterations and set its solution as the true values. The true solutions are 
\begin{equation*}
	\begin{split}
		P^*&=\begin{bmatrix}
			232.2887&   59.3007\\
			59.3007 &  34.5712\\
		\end{bmatrix},
		K^*=\begin{bmatrix}
			59.3007 &  34.5712
		\end{bmatrix},\\
	Y^*&=\begin{bmatrix}
		207.1460 &  56.5767\\
		56.5767 &  33.9800\\
	\end{bmatrix},
		K^*_Y=\begin{bmatrix}
			56.5767 &  33.9800
		\end{bmatrix}.\\
	\end{split}
\end{equation*}
 
Now we devote ourselves to solving the above problem by Algorithm \ref{data-driven PI}, which does not rely on the information of matrices $A$, $B$ and $C$. The parameters of Algorithm \ref{data-driven PI} are set as $d=20$, $s_0=0$, $s_{j+1}=s_j+0.1$, $j=0,1,2,\cdots, d-1$, $K^0=K_Y^0=[35,25]$ and $\varepsilon=10^{-3}$. The initial state is set as $[1,1]^T$. During the simulation, we choose any agent $i$ and randomly collect $10^6$ samples of its state and control policy data.  We set $u_i(t)=-K^0x_i(t)+0.3\sum_{r=1}^{100}\sin(\beta_r t)$, where $\{\beta_r\}_{r=1}^{100}$ is a set of real numbers randomly selected in $[-1000, 1000]$. Then the collected information is used to construct system (\ref{system2}) by the Monte Carlo method. The  input and state trajectories of system (\ref{system2}) used in the simulation are depicted in Figs. 1 and 2, respectively.  

\begin{figure}[ht]
	\begin{minipage}[t]{0.5\linewidth}
		\centering
		\includegraphics[width=\textwidth]{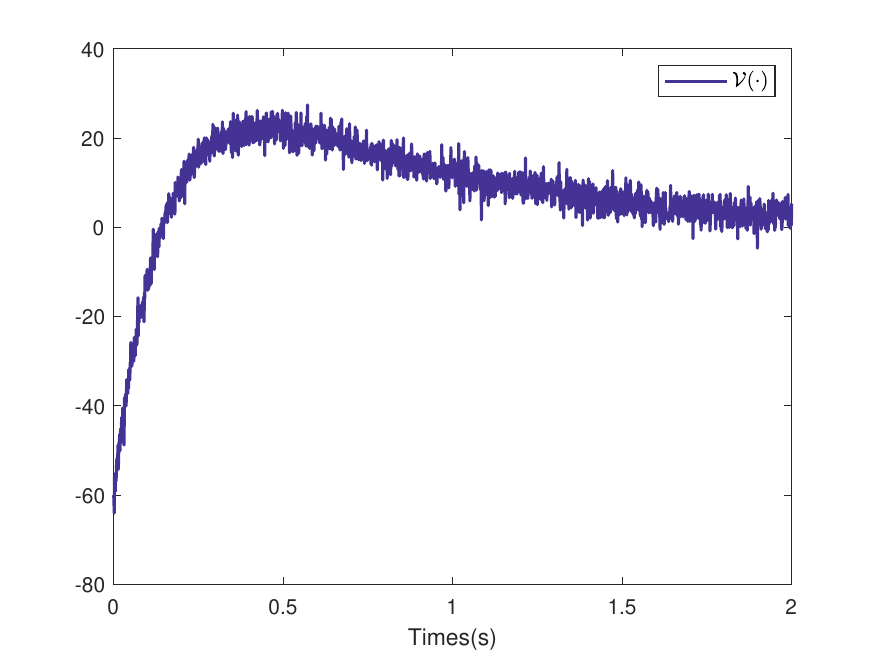}
		\caption{Input trajectory of system (7) used in Section \ref{section5.1}.}
	\end{minipage}%
	\hfill
	\begin{minipage}[t]{0.5\linewidth}
		\centering
		\includegraphics[width=\textwidth]{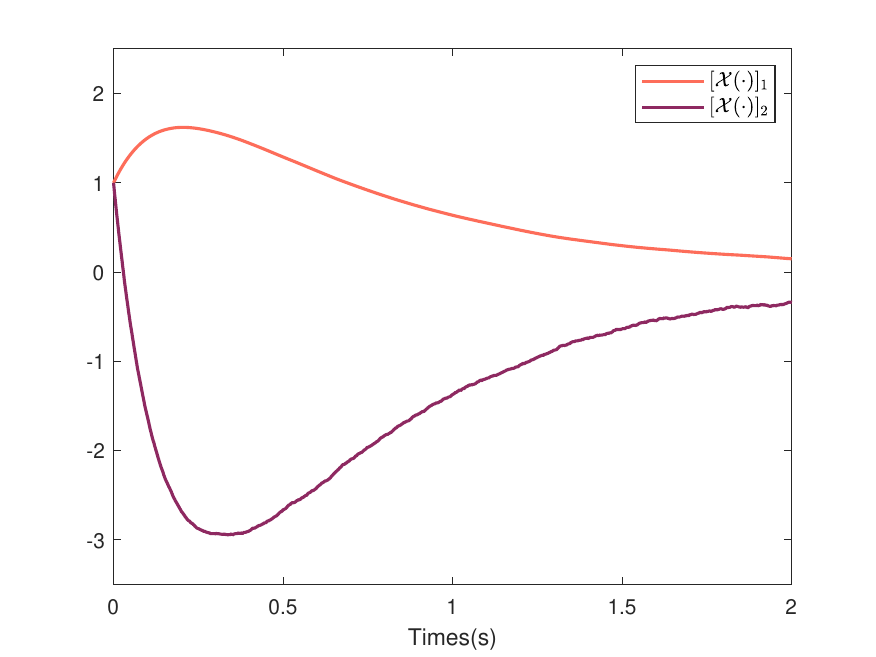}
		\caption{State trajectories of system (7) used in Section \ref{section5.1}.}
	\end{minipage}
\end{figure}

\begin{figure}[!h]
	\begin{minipage}[t]{0.5\linewidth}
		\centering
		\includegraphics[width=\textwidth]{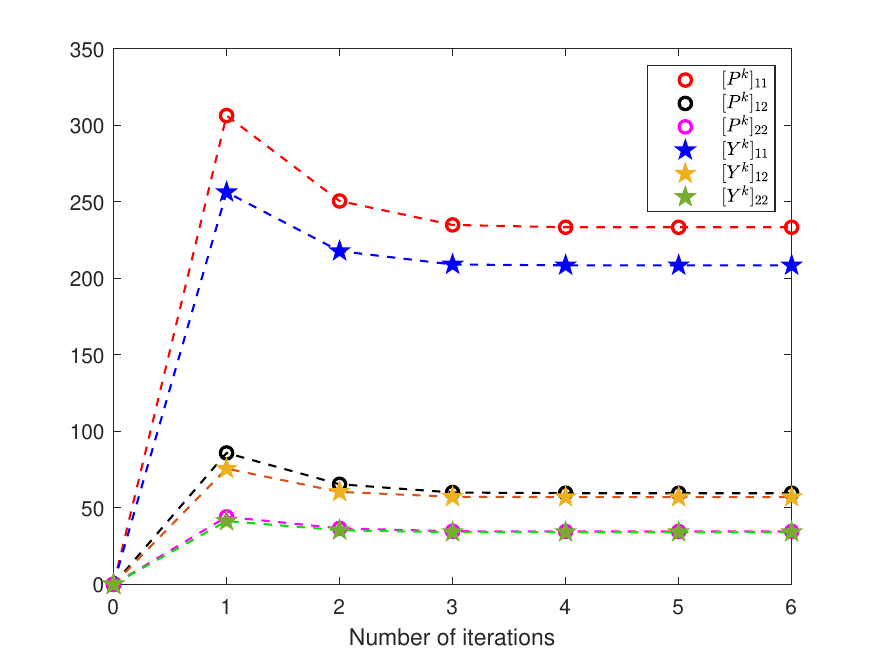}
		\caption{Convergence of matrices $P^k$ and $Y^k$ in Section \ref{section5.1}.}
	\end{minipage}%
	\hfill
	\begin{minipage}[t]{0.5\linewidth}
		\centering
		\includegraphics[width=\textwidth]{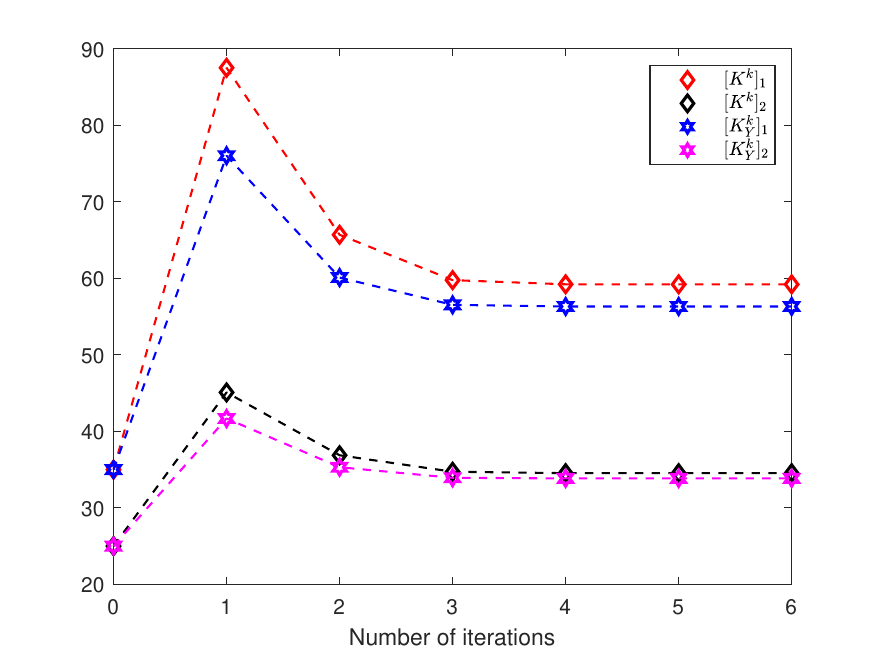}
		\caption{Convergence of matrices $K^k$ and $K^k_Y$ in Section \ref{section5.1}.}
	\end{minipage}
\end{figure}

Figs. 3 and 4 display the convergence of Algorithm \ref{data-driven PI}. After 6 iteration steps, Algorithm \ref{data-driven PI} gives the approximate values of $(P^*,K^*)$ and $(P^*_Y,K^*_Y)$. The results of our algorithm are  
\begin{equation*}
	\begin{split}
		P^6&=\begin{bmatrix}
			233.4279 &  59.5896\\
			59.5896  & 34.5868\\
		\end{bmatrix},
		K^6=\begin{bmatrix}
			59.2254  & 34.5373
		\end{bmatrix},\\
		Y^6&=\begin{bmatrix}
			208.5480 &  57.0425\\
			57.0425 &  34.0963\\
		\end{bmatrix},
		K^6_Y=\begin{bmatrix}
			56.3265 &  33.8537
		\end{bmatrix}.\\
	\end{split}
\end{equation*}

Clearly, the matrices solved by Algorithm \ref{data-driven PI} are close to the true values. This is in good agreement with our theortical results. For comparison purpose, the numerical results of Xu et al. \cite{XuShenHuang2023} are presented in Table 1. It is evident to see from this table that our algorithm has similar performance to theirs. However, as mentioned in Remark 1, the computational complexity of our algorithm is smaller than that of the algorithm in Xu et al. \cite{XuShenHuang2023}.

\begin{table}[H] 
	\centering
	\caption{Comparison between the algorithm in Xu et al. \cite{XuShenHuang2023} and Algorithm \ref{data-driven PI}.} 
	\tabcolsep 27pt
	\begin{tabular*}{\textwidth}{ccc} 
		\hline
		& The algorithm in Xu et al. \cite{XuShenHuang2023} & Algorithm \ref{data-driven PI} \\ 
		\hline
		Final iteration numbers& 6 & 6 \\ 
		\hline
		Relative errors $\frac{|K^6-K^*|}{|K^*|}$ &0.0013 & 0.0012  \\ 
		\hline
		Relative errors $\frac{|K^6_Y-K^*_Y|}{|K^*_Y|}$& 0.0014  & 0.0042 \\ 
		\hline
	\end{tabular*}
\end{table} 

To close this subsection, we focus on the decentralized strategies constructed by the solution of Algorithm \ref{data-driven PI}. Suppose that there are $200$ agents and they adopt decentralized strategies $u_i(\cdot)=-K^6x_i(\cdot)-(K_Y^6-K^6)\widehat{x}_{100}(\cdot)$, $1\leq i\leq 200$, where $\widehat{x}_{100}(\cdot)$ denotes the trajectory of aggregate quantity (\ref{quantity}) simulated by Monte Carlo with $100$ samples. All agents' initial states are randomly generated from $[0,2]\times[0,2]$. Figs. 5 and 6 illustrate the behaviors of the polulation. Furthermore, we plot all agents' average state, which is denoted by  $\widetilde{x}_{200}(\cdot)$,  and the trajectory of $\widehat{x}_{100}(\cdot)$ in Fig. 7. The lines in Fig. 7 show that $\widetilde{x}_{200}(\cdot)$ is close to $\widehat{x}_{100}(\cdot)$, which effectively demonstrate the consistency condition.

\begin{figure}[ht]
	\begin{minipage}[t]{0.5\linewidth}
		\centering
		\includegraphics[width=\textwidth]{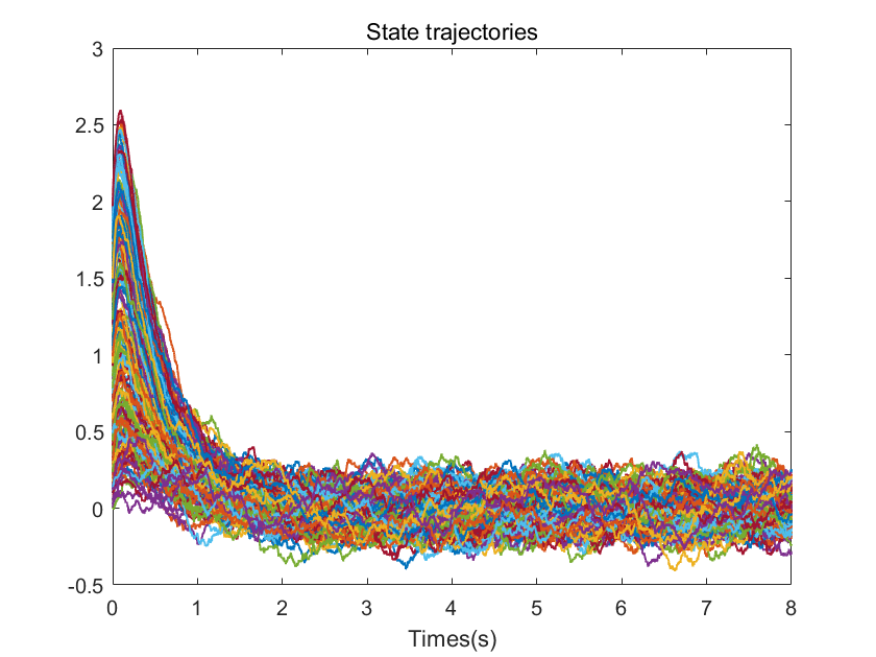}
		\caption{State trajectories $\big[x_i(\cdot)\big]_1$, $1\leq i\leq 200$.}
	\end{minipage}%
	\hfill
	\begin{minipage}[t]{0.5\linewidth}
		\centering
		\includegraphics[width=\textwidth]{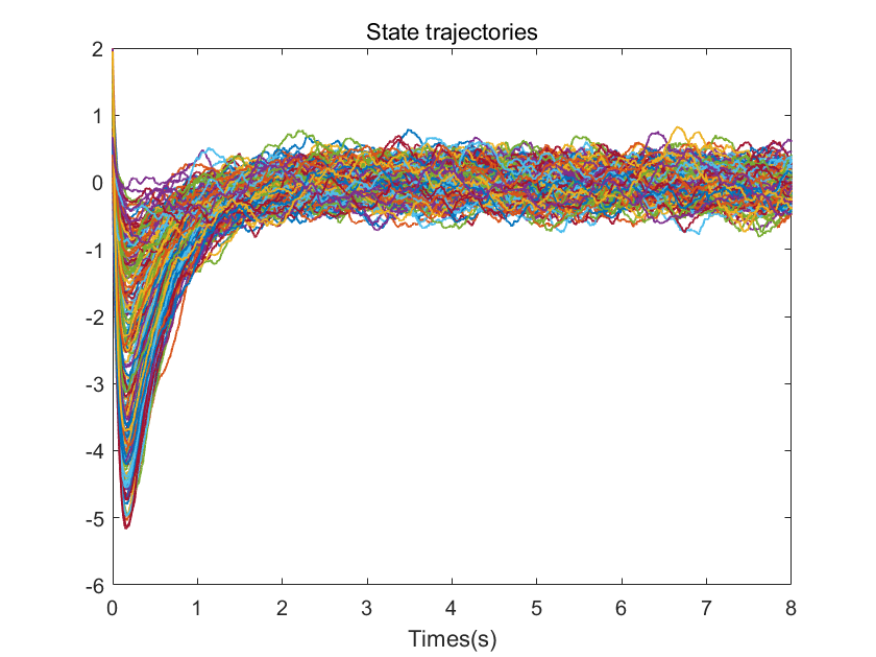}
		\caption{State trajectories $\big[x_i(\cdot)\big]_2$, $1\leq i\leq 200$.}
	\end{minipage}
\end{figure}

\begin{figure}[!h]
		\centering
		\includegraphics[width=0.5\textwidth]{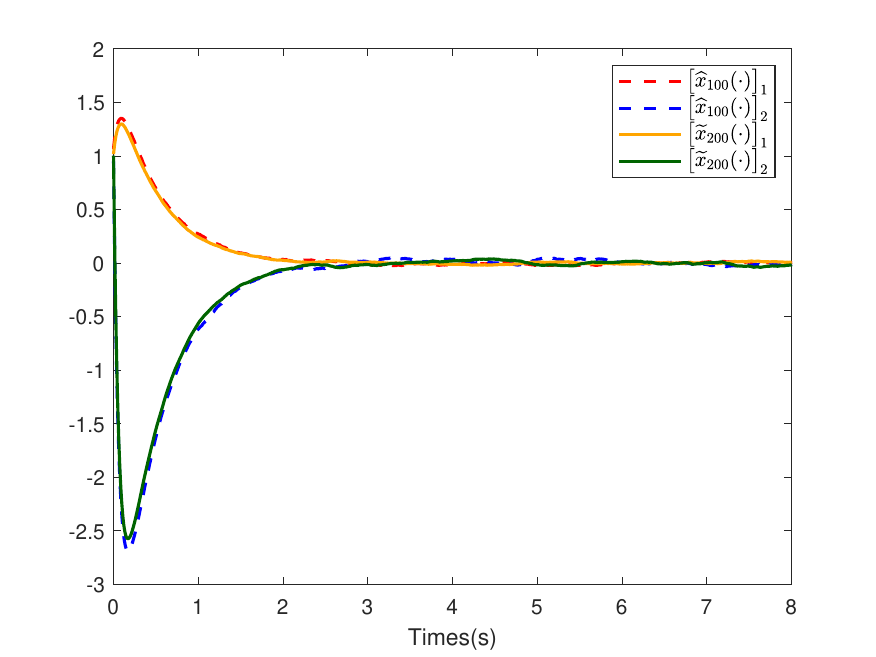}
		\caption{Trajectories of aggregate quantity $\widehat{x}_{100}(\cdot)$ and agents' average state $\widetilde{x}_{200}(\cdot)$.}
\end{figure} 

\subsection{The second numerical example}\label{sec5.2}
As it appears evident,  the first example does not satisfy Assumption \ref{assu2}. In this subsection, we consider an example that satisfies Assumptions \ref{assu1}-\ref{assu2}. The system coefficients are 
\begin{equation*}
	A=\begin{bmatrix}
		 -5 &   1 &        -0.0751\\
		0  & -0.6250 & -39.2699\\
		-0.0045    &     0 &  -0.4127\\
	\end{bmatrix},
	B=\begin{bmatrix}
		1.4542\\
		-0.0154\\
		0.4127\\
	\end{bmatrix},
	C=\begin{bmatrix}
		3 & 0.1\\
		0.5&-2\\
		1&0\\
	\end{bmatrix}.
\end{equation*}
 
The parameters in cost functional are $Q=diag\{5,1,1\}$, $R=1$ and $\rho=0.01$. Since $Y^*=0$ and $K_Y^*=0$, we only focus on calculating the approximate solutions of $P^*$ and $K^*$. For comparison purpose, we present the true values of $P^*$ and $K^*$, which are 
\begin{equation*}
	\begin{split}
		P^*&=\begin{bmatrix}
			0.4976 &   0.1185 &  -1.3229\\
			0.1185  &  0.3377 &  -2.5877\\
			-1.3229 &  -2.5877 &  36.5204\\
		\end{bmatrix},
		K^*=\begin{bmatrix}
			0.1758 &  -0.9008 &  13.1881\\
		\end{bmatrix}.
	\end{split}
\end{equation*}

Next, we will solve this example by Algorithm \ref{data-driven PI} and Algorithm \ref{data-driven VI}, respectively.  To implement Algorithm \ref{data-driven PI}, we set $x_{i0}=[-1,0,1]^T$, $K^0=[-1,-1,14]$, $u_i(t)=-K^0x_i(t)+\sin(-24.6 t)$  and collect $10^{6}$ data samples of a given agent $i$. For Algorithm  \ref{data-driven VI}, we set  $\mathcal{P}^0=0.1\mathbb{I}_n$, $\gamma_k=3/(k+1)$, $u_i(t)=\sin(-6 t)$, $\mathcal{D}_q=\left\{\mathcal{P}\in\mathbb{S}^n_+\big||\mathcal{P}|\leq 100(q+1)\right\}$, $\forall q\in\mathbb{Z}$ and collect $2\times10^{6}$ data samples of agent $i$. The other parameters are set the same as those in Section \ref{section5.1}. The convergence of Algorithms \ref{data-driven PI} and \ref{data-driven VI} is shown in Figs. 8 and 9. 

\begin{figure}[ht]
	\begin{minipage}[t]{0.5\linewidth}
		\centering
		\includegraphics[width=\textwidth]{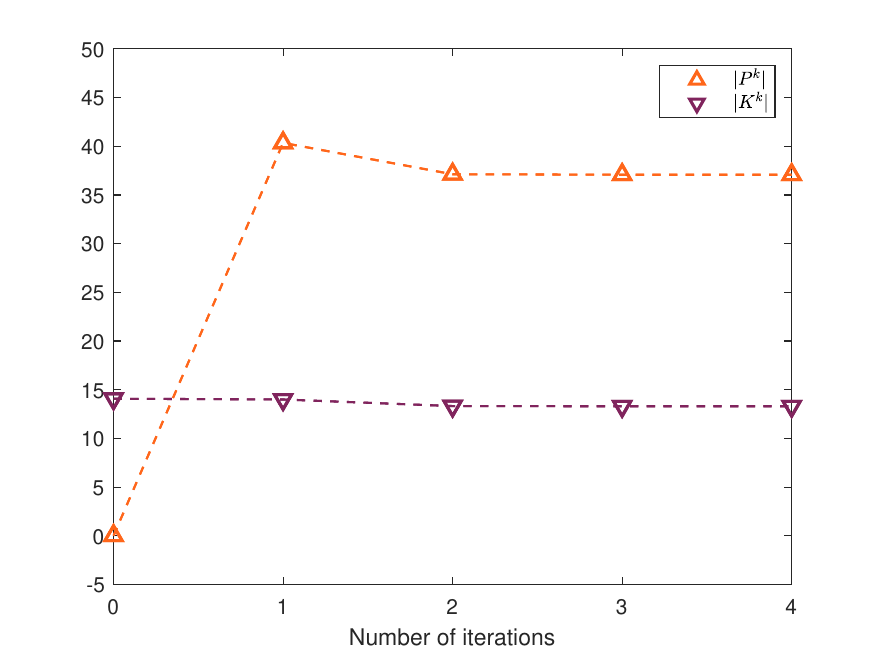}
		\caption{Convergence of Algorithm \ref{data-driven PI} in Section \ref{sec5.2}.}
	\end{minipage}%
	\hfill
	\begin{minipage}[t]{0.5\linewidth}
		\centering
		\includegraphics[width=\textwidth]{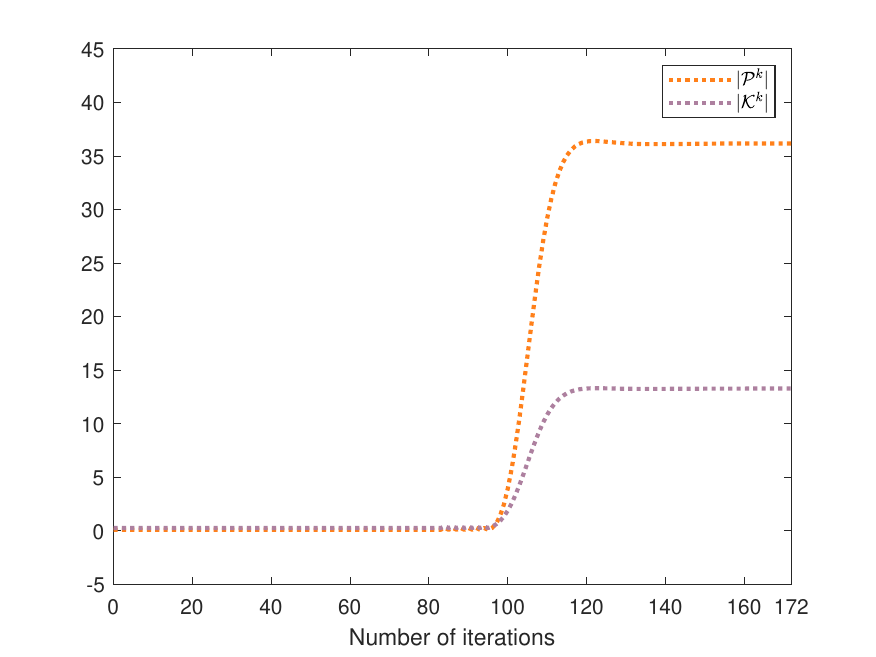}
		\caption{Convergence of Algorithm \ref{data-driven VI} in Section \ref{sec5.2}}
	\end{minipage}
\end{figure}

After 4 iteration steps, Algorithm \ref{data-driven PI} generates the following solution matrices 
\begin{equation*}
	\begin{split}
		P^4&=\begin{bmatrix}
			0.3825  &  0.1836  &  -1.1736\\
			0.1836  &  0.2990  &  -2.6751\\
			-1.1736 &  -2.6751 &   36.8479\\
		\end{bmatrix},
		K^4=\begin{bmatrix}
			0.1416 &  -0.8537 &  13.2655\\
		\end{bmatrix}
	\end{split}
\end{equation*}
with a relative error $|K^4-K^*|/|K^*|=0.0073$. Moreover, Algorithm \ref{data-driven VI} converges after 172 iterations and gives
\begin{equation*}
	\begin{split}
		\mathcal{P}^{172}&=\begin{bmatrix}
			0.3980 &   0.1251  & -1.1929\\
			0.1251 &   0.3397  & -2.5855\\
			-1.1929 &  -2.5855 &  35.9500\\
		\end{bmatrix},
		\mathcal{K}^{172}=\begin{bmatrix}
			0.1524 &  -0.8951 &  13.2542\\
		\end{bmatrix}
	\end{split}
\end{equation*} 
with its relative error  $|\mathcal{K}^{172}-\mathcal{K}^*|/|\mathcal{K}^*|=0.0053$.  It is worth pointing out that the the straight dashed lines in the first half of Fig. 9 indicate that $\mathcal{P}^k\notin\mathcal{D}_q$ in the early iteration steps. This is one of the characteristics of continuous-time VI algorithms. 

According to the above simulation results, it follows that the algorithms developed in this paper may be useful in dealing with LQ mean-field game problems under the setting of completely unknown system coefficients.  Therefore, the proposed algorithm may be more conducive to solving practical mean-field game application problems.

\section{Conclusions and  future works}
This paper is concerned with an LQ mean-field game problem in continuous-time. We develop a system transformation method to implement a model-based PI algorithm and a model-based VI algorithm. The obtained data-driven algorithms permit the construction of decentralized control strategies without system coefficient information and have smaller computational complexities. Moreover, we simulate two numerical examples to verify the effectiveness of our algorithms. The simulation results show that our algorithms successfully find the  $\epsilon$-Nash equilibria with small errors. Thus, the algorithms proposed in this paper may be promising tools in solving continuous-time LQ mean-field games with unknown system parameters. In future works, we want to explore data-driven RL algorithms for  more general LQ mean-field games such as those with jumps, delays and partial information.

\section*{Acknowledgements}
\noindent X. Li acknowledges the financial support by the Hong Kong General Research Fund, China under Grant Nos. 15216720, 15221621 and 15226922. G. Wang acknowledges the financial support from the National Key R\&D Program of China under Grant No. 2022YFA1006103, the National Natural Science Foundation of China under Grant Nos. 61925306, 61821004 and 11831010, and the Natural Science Foundation of Shandong Province under Grant No. ZR2019ZD42. J. Xiong acknowledges the financial support from the National Key R\&D Program of China under Grant No. 2022YFA1006102 and the National Natural Science Foundation of China under Grant No. 11831010.



	
\end{document}